\documentclass[11pt,reqno]{amsart}
\usepackage{amsmath,amssymb}

 \makeatletter
 \evensidemargin  \oddsidemargin
 \makeatother

\begin{document}
\thispagestyle{empty}
 \setcounter{page}{1}

\begin{center}
{\large\bf Stability of homomorphisms and derivations in
$C^*$-ternary algebras

\vskip.20in

{\bf M. Bavand Savadkouhi, M. Eshaghi Gordji and N. Ghobadipour  } \\[2mm]

{\footnotesize Department of Mathematics,
Semnan University,\\ P. O. Box 35195-363, Semnan, Iran\\
[-1mm] E-mail: {\tt bavand.m@gmail.com , madjid.eshaghi@gmail.com
\\ghobadipour.n@gmail.com}}}
\end{center}
\vskip 5mm \noindent{\footnotesize{\bf Abstract.} In this paper,
we investigate homomorphisms between $C^*$-ternary algebras and
derivations on $C^*$-ternary algebras, associated with the
following functional equation
$$f(\frac{x_2-x_1}{3})+f(\frac{x_1-3x_3}{3})+f(\frac{3x_1+3x_3-x_2}{3})=f(x_1).$$
Moreover, we prove the generalized Hyers-Ulam -Rassias  stability
of homomorphisms in $C^*$-ternary algebras and of derivations on
$C^*$-ternary algebras.
 }

\vskip.10in

  \newtheorem{df}{Definition}[section]
  \newtheorem{rk}[df]{Remark}
   \newtheorem{lem}[df]{Lemma}
   \newtheorem{thm}[df]{Theorem}
   \newtheorem{pro}[df]{Proposition}
   \newtheorem{cor}[df]{Corollary}
   \newtheorem{ex}[df]{Example}

 \setcounter{section}{0}
 \numberwithin{equation}{section}

\vskip .2in
\begin{center}
\section{Introduction}
\end{center}
Ternary algebraic operations were considered in the 19 th century by
several mathematicians and physicists such as Cayley [9] who
introduced the notions the of cubic matrix, which in turn was
generalized by Kapranov at el.[23]. The simplest example of such
nontrivial ternary operation is given by the following composition
rule:
$$\{a,b,c\}_{ijk}=\sum_{l,m,n}a_{nil} b_{ljm} c_{mkn}~.\hspace{.75 cm}(i,j,k,...=1,2,...,\Bbb {N})$$
Ternary structures and their generalization, the so-called n-ary
structures, raise certain hopes in view of their applications in
physics. Some significant physical applications are as follows
(see Refs.[24] and [25]):\\ (i) The algebra of "nonions" generated
by
two matrices,\\
$\left(%
\begin{array}{ccc}
  0 & 1 & 0 \\
  0 & 0 & 1 \\
  1 & 0 & 0 \\
\end{array}%
\right)$
and
$\left(%
\begin{array}{ccc}
  0 & 1 & 0 \\
  0 & 0 & \omega \\
  \omega^2 & 0 & 0 \\
\end{array}%
\right)~, \hspace{.75 cm} (\omega=e^{\frac{2\Pi i}{3}}) $\\
was introduced by Sylvester as a ternary analog of Hamilton's
quaternions (cf. Ref. [1]).\\ (ii) A natural ternary composition
of $4$-vectors in the four-dimensional Minkowskian space time
$M_4$ can be defined as an example of a ternary operation:
$$(X,Y,Z)\longmapsto U(X,Y,Z) \in M_4~,$$
with the resulting 4-vector $U^{\mu}$ defined via its components
in a given coordinate system as follows:
$$U^{\mu}(X,Y,Z)=g^{\mu\sigma}\eta_{\sigma\nu\lambda\rho}X^{\nu}Y^{\lambda}Z^{\rho}~, \hspace{0.75 cm} \mu,\nu, . . . =
0,1,2,3,$$ where $g^{\mu\sigma}$ is the metric tensor and
$\eta_{\sigma\nu\lambda\rho}$ is the canonical volume element of
$M_4$.~[25]\\ (iii) The quark model inspired a particular brand of
ternary algebraic systems. The "Nambu mechanics" is based on such
structures (see Refs. [11] and [52]). Quarks apparently couple by
packs of 3.\\
There are also some applications, although still hypothetical, in
the fractional quantum Hall effect, nonstandard statistics,
supersymmetric theory, Yang–Baxter equation, etc., cf. Refs. [1],
[25], and [54]. Following the terminology of Ref. [12], a nonempty
set G with a ternary operation $[.,.,.]:G^3\longrightarrow G$ is
called a ternary groupoid and is denoted by $(G,[.,.,.]).$ The
ternary groupoid $(G,[.,.,.])$ is called commutative if
$[x_1,x_2,x_3]=[x_{\sigma(1)},x_{\sigma(2)},x_{\sigma(3)}]$ for all
$x_1,x_2,x_3 \in G$ and all permutations $\sigma$ of $\{1,2,3\}.$ If
a binary operation $\circ$ is defined on G such that $[x,y,z]=(x
\circ y)\circ z$ for all $x,y,z \in G$, then we say that $[.,.,.]$
is derived from $\circ$. We say that $(G,[.,.,.])$ is a ternary
semigroup if the operation $[.,.,.]$ is associative, i.e., if
$[[x,y,z],u,v]=[x,[y,z,u],v]=[x,y,[z,u,v]]$ holds for all $x,y,z,u,v
\in G$ (see Ref. [8]).\\ As it is extensively discussed in [50], the
full description of a physical system $ \Bbb S$ implies the
knowledge of three basis ingredients: the set of the observables,
the set of the states and the dynamics that describes the time
evolution of the system by means of the time dependence of the
expectation value of a given observable on a given statue.
Originally the set of the observable was considered to be a
$C^*-$algebra [17]. In many applications, however, this was shown
not to be the must convenient choice and the $C^*-$algebra was
replaced by a von Neumann algebra, because the role of the
representation terns out to be crucial mainly when long range
interactions are involved (see [6] and references therein). Here we
used a different algebraic structure. A $C^*$-ternary algebra is a
complex Banach space A, equipped with a ternary product
$(x,y,z)\rightarrowtail[x,y,z]$ of $A^3$ into $A,$ which is $\Bbb
C$-linear in the outer variables, conjugate $\Bbb C$-linear in the
middle variable, and associative in the sense that
$[x,y,[z,w,v]]=[x,[w,z,y],v]=[[x,y,z],w,v],$ and satisfies
$\|[x,y,z]\|\leq\|x\|.\|y\|.\|z\|$ and $\|[x,x,x]\|=\|x\|^3$
(see[27]). If a $C^*$-ternary algebra $(A,[.,.,.])$ has an identity,
i.e., an element $e\in A$ such that $x=[x,e,e]=[e,e,x]$ for all
$x\in A,$ then it is routine to verify that $A,$ endowed with
$xoy:=[x,e,y]$ and $x^*:=[e,x,e],$ is a unital $C^*$- algebra.
Conversely, if $(A,o)$ is a unital $C^*$- algebra, then
$[x,y,z]:=xoy^*oz$ makes $A$ into a $C^*$-ternary algebra. A $\Bbb
C$-linear mapping $H:A \to B$ is called a $C^*$-ternary algebra
homomorphism if $$H([x,y,z])=[H(x),H(y),H(z)]$$ for all $x,y,z\in
A.$  A $\Bbb C$-linear mapping $\delta:A \to A$ is called a
$C^*$-ternary algebra derivation if
$$\delta([x,y,z])=[\delta(x),y,z]
+[x,\delta(y),z]+[x,y,\delta(z)]$$ for all $x,y,z\in A.$

Ternary structures and their generalization the so-called $n$-ary
structures, raise certain hops in view of their applications in
physics  (see [2-4], [6], [17], [24,26,27], [30], [50] and [55]).

The study of stability problems originated from a famous talk
given by S. M. Ulam [53] in 1940:"Under what condition dose there
exists a homomorphism near an approximate homomorphism?" In the
next year 1941, D. H. Hyers [19] was answered affirmatively the
question of Ulam and the result can be formulated as follows: if
$\epsilon>0$ and $f:{E_1}\longrightarrow{E_2}$ is  a map with
$E_1$ a normed space, $E_2$ a Banach spaces such that
$$\|f(x+y)-f(x)-f(y)\|\leq \epsilon $$
for all $x,y\in E_1,$  then there exists a unique additive map
$T:{E_1}\longrightarrow{E_2}$ such that
$$\|f(x)-T(x)\|\leq \epsilon$$
for all $x\in E_1.$ Moreover, if $f(tx)$ is continuous in $t\in
\Bbb R$ for each fixed $x\in E_1,$ then $T$ is linear. This
stability phenomenon is called the Hyers-Ulam stability of the
additive functional equation $g(x+y)=g(x)+g(y).$ A generalized
version of the theorem of Hyers for approximately additive maps
was given by Th. M. Rassias [46] in 1978 by considering the case
when the above inequality is not bounded:
\begin{thm}\label{t1} Let $f:{E_1}\longrightarrow{E_2}$ be a mapping from
a normed vector space ${E_1}$ into a Banach space ${E_2}$ subject
to the inequality
$$\|f(x+y)-f(x)-f(y)\|\leq \epsilon (\|x\|^p+\|y\|^p) $$
for all $x,y\in E_1,$ where $\epsilon$ and p are constants with
$\epsilon>0$ and $p<1.$ Then there exists a unique additive
mapping $T:{E_1}\longrightarrow{E_2}$ such that
$$\|f(x)-T(x)\|\leq \frac{2\epsilon}{2-2^p}\|x\|^p ,$$ for all $x\in E_1.$
\end{thm}
The stability phenomenon that was introduced and proved by  Th. M.
Rassias is called Hyers-Ulam-Rassias stability. And then the
stability problems of several functional equations have been
extensively investigated by a number of authors and there are many
interesting results concerning this problem. (see [5],
[7],[10],[13-16], [18-22], [28,29], [31-49] and [51]).

Throughout this paper,  we assume that $A$ is a $C^*$-ternary
algebra with norm $\|.\|_A$ and that $B$ is a $C^*$-ternary
algebra with norm $\|.\|_B.$\\

\vskip 5mm
\begin{center}
\section{Superstability  of Homomorphisms and derivations on $C^*$-ternary algebras }
\end{center}
In this section, first we investigate homomorphisms between
$C^*$-ternary algebras. We need the following Lemma in the  main results of the paper. \\
\begin{lem}\label{t2}
Let $f:A \to B$ be a mapping such that
$$\|f(\frac{x_2-x_1}{3})+f(\frac{x_1-3x_3}{3})+f(\frac{3x_1+3x_3-x_2}{3})\|_B \leq \|f(x_1)\|_B, \eqno(2.1)$$
for all $x_1,x_2,x_3 \in A.$ Then $f$ is additive.
\end{lem}
\begin{proof}
Letting $x_1=x_2=x_3=0$ in $(2.1),$ we get $$\|3f(0)\|_B \leq
\|f(0)\|_B.$$So $f(0)=0.$ Letting $x_1=x_2=0$ in $(2.1),$ we get
$$\|f(-x_3)+f(x_3)\|_B \leq \|f(0)\|_B=0$$ for all $x_3 \in A.$
Hence $f(-x_3)=-f(x_3)$ for all $x_3 \in A.$ Letting $x_1=0$ and
$x_2=6x_3$ in $(2.1),$ we get $$\|f(2x_3)-2f(x_3)\|_B \leq
\|f(0)\|_B=0$$ for all $x_3 \in A.$ Hence $$f(2x_3)=2f(x_3)$$ for
all $x_3 \in A.$ Letting $x_1=0$ and $x_2=9x_3$ in $(2.1),$ we get
$$\|f(3x_3)-f(x_3)-2f(x_3)\|_B \leq \|f(0)\|_B=0$$ for all $x_3 \in
A.$Hence $$f(3x_3)=3f(x_3)$$ for all $x_3 \in A.$Letting $x_1=0$
in $(2.1),$ we get
$$\|f(\frac{x_2}{3})+f(-x_3)+f(x_3-\frac{x_2}{3})\|_B \leq
\|f(0)\|_B=0$$ for all $x_2,x_3 \in A.$ So
$$f(\frac{x_2}{3})+f(-x_3)+f(x_3-\frac{x_2}{3})=0\eqno(2.2)$$ for all $x_2,x_3 \in
A.$ Let $t_1=x_3-\frac{x_2}{3}$ and $t_2=\frac{x_2}{3}$ in
$(2.2).$ Then $$f(t_2)-f(t_1+t_2)+f(t_1)=0$$ for all $t_1,t_2 \in
A$ and so $f$ is additive.
\end{proof}

\begin{thm}\label{t2}
Let $p\neq 1$ and $\theta$ be nonnegative real numbers, and let
$f:A \to B$ be a mapping such that
$$\|f(\frac{x_2-x_1}{3})+f(\frac{x_1-3\mu x_3}{3})+\mu f(\frac{3x_1+3x_3-x_2}{3})\|_B \leq \|f(x_1)\|_B,\eqno(2.3)$$
$$\|f([x_1,x_2,x_3])-[f(x_1),f(x_2),f(x_3)]\|_B \leq \theta (\|x_1\|_A^{3p}+\|x_2\|_A^{3p}+\|x_3\|_A^{3p})\eqno(2.4)$$
for all $\mu \in \Bbb T^1:=\{\lambda \in \Bbb C~; |\lambda|=1\}$
and all $x_1,x_2,x_3 \in A.$ Then the mapping $f:A \to B$ is a
$C^*$-ternary algebra homomorphism.
\end{thm}
\begin{proof}
Assume $p > 1.$\\
Let $\mu=1$ in $(2.3).$ By lemma 2.1, the mapping $f:A \to B$ is
additive. Letting $x_1=x_2=0$ in $(2.3),$ we get $$\|f(-\mu
x_3)+\mu f(x_3)\|_B \leq \|f(0)\|_B=0$$ for all $x_3\in A$ and
$\mu \in \Bbb T^1.$ So $$-f(\mu x_3)+\mu f(x_3)=f(-\mu x_3)+\mu
f(x_3)=0$$ for all $x_3\in A$ and all $\mu \in \Bbb T^1.$ Hence
$f(\mu x_3)=\mu f(x_3)$ for all $x_3 \in A$ and all $\mu \in \Bbb
T^1.$ By the theorem 2.1 of [33], the mapping $f:A \to B$ is $\Bbb
C$-linear. It follows from $(2.4)$ that
\begin{align*}
&\|f([x_1,x_2,x_3])-[f(x_1),f(x_2),f(x_3)]\|_B\\
&=\lim_{n \to \infty} 8^n
\|f(\frac{[x_1,x_2,x_3]}{2^n.2^n.2^n})-[f(\frac{x_1}{2^n}),f(\frac{x_2}{2^n}),f(\frac{x_3}{2^n})]\|_B\\
&\leq \lim_{n \to \infty}\frac{8^n
\theta}{8^{np}}(\|x_1\|_A^{3p}+\|x_2\|_A^{3p}+\|x_3\|_A^{3p})\\
&=0
\end{align*}
for all $x_1,x_2,x_3 \in A.$ Thus
$$f([x_1,x_2,x_3])=[f(x_1),f(x_2),f(x_3)]$$ for all $x_1,x_2,x_3 \in
A.$ Hence the mapping $f:A \to B$ is a $C^*$-ternary algebra
homomorphism. Similarly, one obtains the result for the case
$p<1.$

\end{proof}

Now we establish the superstability of derivations from a
$C^*$-ternary algebra into  its $C^*$-ternary modules as follows.
\begin{thm}
Let $p\neq 1$ and $\theta$ be nonnegative real numbers, and let
$f:A \to A$ be a mapping satisfying $(2.3)$ such that
\begin{align*}
&\|f([x_1,x_2,x_3])-[f(x_1),x_2,x_3]-[x_1,f(x_2),x_3]-[x_1,x_2,f(x_3)]\|_A\\
&\leq \theta
(\|x_1\|_A^{3p}+\|x_2\|_A^{3p}+\|x_3\|_A^{3p})\hspace{7 cm} (2.5)
\end{align*}
for all $x_1,x_2,x_3 \in A.$ Then the mapping $f:A \to A$ is a
$C^*$-ternary derivation.
\end{thm}
\begin{proof}
Assume $p > 1.$\\
By the theorem 2.2, the mapping $f:A \to A$ is $\Bbb C$-linear. It
follows from $(2.5)$ that
\begin{align*}
&\|f([x_1,x_2,x_3])-[f(x_1),x_2,x_3]-[x_1,f(x_2),x_3]-[x_1,x_2,f(x_3)]\|_A\\
&=\lim_{n \to
\infty}8^n\|f(\frac{[x_1,x_2,x_3]}{8^n})-[f(\frac{x_1}{2^n}),\frac{x_2}{2^n},\frac{x_3}{2^n}]-[\frac{x_1}{2^n},f(\frac{x_2}{2^n}),\frac{x_3}{2^n}]\\
&-[\frac{x_1}{2^n},\frac{x_2}{2^n},f(\frac{x_3}{2^n})]\|_A\\
&\leq \lim_{n \to \infty}\frac{8^n
\theta}{8^{np}}(\|x_1\|_A^{3p}+\|x_2\|_A^{3p}+\|x_3\|_A^{3p})\\
&=0
\end{align*}
for all $x_1,x_2,x_3 \in A.$ So
$$f([x_1,x_2,x_3])=[f(x_1),x_2,x_3]+[x_1,f(x_2),x_3]+[x_1,x_2,f(x_3)]$$
for all $x_1,x_2,x_3 \in A.$ Thus the mapping $f:A \to A$ is a
$C^*$-ternary derivation. Similarly, one obtains the result for
the case $p<1.$
\end{proof}

\vskip 5mm
\begin{center}
\section{Stability of homomorphisms and derivations on $C^*$-ternary algebras}
\end{center}
First we prove the generalized Hyers-Ulam -Rassias stability of
homomorphisms in $C^*$-ternary algebras.
\begin{thm}
Let $p>1$ and $\theta$ be nonnegative real numbers, and let $f:A
\to B$ be a mapping such that
\begin{align*}
&\|f(\frac{x_2-x_1}{3})+f(\frac{x_1-3\mu x_3}{3})+\mu
f(\frac{3x_1+3x_3-x_2}{3})-f(x_1)\|_B \\
&\leq \theta(\|x_1\|_A^p+\|x_2\|_A^p+\|x_3\|_A^p)\hspace{7
cm}(3.1)
\end{align*}
and
$$\|f([x_1,x_2,x_3])-[f(x_1),f(x_2),f(x_3)]\|_B \leq \theta (\|x_1\|_A^{3p}+\|x_2\|_A^{3p}+\|x_3\|_A^{3p})\eqno(3.2)$$
for all $\mu \in \Bbb T^1$ and all $x_1,x_2,x_3 \in A.$ Then there
exists a unique $C^*$-ternary homomorphism $H:A \to B$ such that
$$\|H(x_1)-f(x_1)\|_B \leq \frac{\theta(1+2^p)\|x_1\|_A^p}{1-3^{1-p}}\eqno(3.3)$$
for all $x_1 \in A.$
\end{thm}
\begin{proof}
Let us assume $\mu=1$ ,$x_2=2x_1$ and $x_3=0$ in $(3.1).$ Then we
get $$\|3f(\frac{x_1}{3})-f(x_1)\|_B \leq
\theta(1+2^p)\|x_1\|_A^p\eqno(3.4)$$ for all $x_1 \in A.$ So by
induction, we have
$$\|3^nf(\frac{x_1}{3^n})-f(x_1)\|_B \leq \theta(1+2^p)\|x_1\|_A^p \sum_{i=0}^{n-1}3^{i(1-p)}\eqno(3.5)$$
for all $x_1 \in A.$ Hence
\begin{align*}
\|3^{n+m}f(\frac{x_1}{3^{n+m}})-3^mf(\frac{x_1}{3^m})\|_B
&\leq \theta (1+2^p)\|x_1\|_A^p \sum_{i=0}^{n-1}3^{(i+m)(1-p)}\\
&\leq \theta (1+2^p)
\|x_1\|_A^p\sum_{i=m}^{n+m-1}3^{i(1-p)}\hspace{1.7 cm} (3.6)
\end{align*}
for all nonnegative integers $m$ and $n$ with $n \geq m$ and all
$x_1 \in A.$ From this it follows that the sequence $\{ 3^n
f(\frac{x_1}{3^n})\}$ is a Cauchy sequence for all $x_1 \in A.$
Since $B$ is complete, the sequence $\{3^n f(\frac{x_1}{3^n})\}$
converges. Thus one can define the mapping $H:A \to B$ by
$$H(x_1):=\lim_{n \to \infty} 3^n f(\frac{x_1}{3^n})$$ for all
$x_1 \in A.$ Moreover, letting $m=0$ and passing the limit $n \to
\infty$ in $(3.6),$ we get $(3.3).$ It follows from $(3.1)$ that
\begin{align*}
&\|H(\frac{x_2-x_1}{3})+H(\frac{x_1-3\mu x_3}{3})+\mu
H(\frac{3x_1+3x_3-x_2}{3})-H(x_1)\|_B\\
&=\lim_{n \to \infty}
3^n\|f(\frac{x_2-x_1}{3^{n+1}})+f(\frac{x_1-3\mu
x_3}{3^{n+1}})+f(\frac{3x_1+3x_3-x_2}{3^{n+1}})-f(\frac{x_1}{3^n})\|_B\\
&\leq \lim_{n \to
\infty}\frac{3^n\theta}{3^{np}}(\|x_1\|_A^p+\|x_2\|_A^p+\|x_3\|_A^p)=0
\end{align*}
for all $\mu \in \Bbb T^1$ and all $x_1,x_2,x_3 \in A.$ So
$$H(\frac{x_2-x_1}{3})+H(\frac{x_1-3\mu x_3}{3})+\mu
H(\frac{3x_1+3x_3-x_2}{3})=H(x_1)$$ for all $\mu \in \Bbb T^1$ and
all $x_1,x_2,x_3 \in A.$ By the Theorem 2.1 of [33], the mapping
$H:A \to B$ is $\Bbb C$-linear.\\
Now, let $H^{'}:A \to B$ be another additive mapping satisfying
$(3.3).$ Then we have
\begin{align*}
&\|H(x_1)-H^{'}(x_1)\|_B=3^n
\|H(\frac{x_1}{3^n})-H^{'}(\frac{x_1}{3^n})\|_B\\
&\leq 3^n
(\|H(\frac{x_1}{3^n})-f(\frac{x_1}{3^n})\|_B+\|H^{'}(\frac{x_1}{3^n})-f(\frac{x_1}{3^n})\|_B)\\
&\leq \frac{2.3^n \theta(1+2^p)}{3^{np}(1-3^{1-p})}\|x\|_A^p~,
\end{align*}
which tends to zero as $n \to \infty$ for all $x_1 \in A.$ So we
can conclude that $H(x_1)=H^{'}(x_1)$ for all $x_1 \in A.$ This
proves the uniqueness of $H.$\\ It follows from $(3.2)$ that
\begin{align*}
&\|H([x_1,x_2,x_3])-[H(x_1),H(x_2),H(x_3)]\|_B\\
&=\lim_{n \to \infty} 27^n
\|f(\frac{[x_1,x_2,x_3]}{3^n.3^n.3^n})-[f(\frac{x_1}{3^n}),f(\frac{x_2}{3^n}),f(\frac{x_3}{3^n})]\|_B\\
&\leq \lim_{n \to \infty}\frac{27^n
\theta}{27^{np}}(\|x_1\|_A^{3p}+\|x_2\|_A^{3p}+\|x_3\|_A^{3p})=0
\end{align*}
for all $x_1,x_2,x_3 \in A.$\\
Thus the mapping $H:A \to B$ is a unique $C^*$-ternary
homomorphism satisfying $(3.3).$
\end{proof}

\begin{thm}
Let $p<1$ and $\theta$ be nonnegative real numbers, and let $f:A
\to B$ be a mapping satisfying $(3.1)$ and $(3.2).$ Then there
exists a unique $C^*$-ternary homomorphism $H:A \to B$ such that
$$\|H(x_1)-f(x_1)\|_B \leq \frac{\theta(1+2^p)\|x_1\|_A^p}{3^{1-p}-1}\eqno(3.7)$$
for all $x_1 \in A.$
\end{thm}
\begin{proof}
The proof is similar to the proof of Theorem 3.1.
\end{proof}

Now we prove the generalized Hyers-Ulam -Rassias stability of
derivations from a $C^*$-ternary algebra into its $C^*$-ternary
moduls.
\begin{thm}
Let $p>1$ and $\theta$ be nonnegative real numbers, and let $f:A
\to A$ be a mapping such that
\begin{align*}
&\|f(\frac{x_2-x_1}{3})+f(\frac{x_1-3\mu x_3}{3})+\mu
f(\frac{3x_1+3x_3-x_2}{3})-f(x_1)\|_A \\
&\leq \theta(\|x_1\|_A^p+\|x_2\|_A^p+\|x_3\|_A^p)\hspace{7
cm}(3.8)
\end{align*}
and
\begin{align*}
&\|f([x_1,x_2,x_3])-[f(x_1),x_2,x_3]-[x_1,f(x_2),x_3]-[x_1,x_2,f(x_3)]\|_A\\
&\leq \theta
(\|x_1\|_A^{3p}+\|x_2\|_A^{3p}+\|x_3\|_A^{3p})\hspace{7 cm}(3.9)
\end{align*}
for all $\mu \in \Bbb T^1$ and all $x_1,x_2,x_3 \in A.$ Then there
exists a unique $C^*$-ternary derivation $D:A \to A$ such that
$$\|D(x_1)-f(x_1)\|_A \leq \frac{\theta(1+2^p)\|x_1\|_A^p}{1-3^{1-p}}\eqno(3.10)$$
for all $x_1 \in A.$
\end{thm}
\begin{proof}
By the same reasoning as in the proof of the Theorem 3.1, there
exists a unique $\Bbb C$-linear mapping $D:A \to A$ satisfying
$(3.10).$ The mapping $D:A \to A$ is defined by
$$D(x_1):=\lim_{n \to \infty}3^n f(\frac{x_1}{3^n})$$ for all $x_1 \in
A.$ It follows from $(3.9)$ that
\begin{align*}
&\|D([x_1,x_2,x_3])-[D(x_1),x_2,x_3]-[x_1,D(x_2),x_3]-[x_1,x_2,D(x_3)]\|_A\\
&=\lim_{n \to
\infty}27^n\|\frac{[x_1,x_2,x_3]}{3^n.3^n.3^n}-[f(\frac{x_1}{3^n}),\frac{x_2}{3^n},\frac{x_3}{3^n}]-[\frac{x_1}{3^n},f(\frac{x_2}{3^n}),\frac{x_3}{3^n}]
-[\frac{x_1}{3^n},\frac{x_2}{3^n},f(\frac{x_3}{3^n})]\|_A\\
&\leq \lim_{n \to
\infty}\frac{27^n\theta}{27^{np}}(\|x_1\|_A^{3p}+\|x_2\|_A^{3p}+\|x_3\|_A^{3p})=0
\end{align*}
for all $x_1,x_2,x_3 \in A.$ So
$$D([x_1,x_2,x_3])=[D(x_1),x_2,x_3]+[x_1,D(x_2),x_3]+[x_1,x_2,D(x_3)]$$
for all $x_1,x_2,x_3 \in A.$\\
Thus the mapping $D:A \to A$ is a unique $C^*$-ternary derivation
satisfying $(3.10).$
\end{proof}

\begin{thm}
Let $p<1$ and $\theta$ be nonnegative real numbers, and let $f:A
\to A$ be a mapping satisfying $(3.8)$ and $(3.9).$ Then there
exists a unique $C^*$-ternary derivation $D:A \to A$ such that
$$\|D(x_1)-f(x_1)\|_A \leq \frac{\theta(1+2^p)\|x_1\|_A^p}{3^{1-p}-1}\eqno(3.11)$$
for all $x_1 \in A.$
\end{thm}
\begin{proof}
The proof is similar to the proof of Theorems 3.1 and 3.3.
\end{proof}

{\small

}
\end{document}